\documentclass{amsart}
\usepackage{amsmath}
\usepackage{amssymb}
\usepackage{amsthm}
\usepackage[msc-links,abbrev]{amsrefs}
\usepackage{hyperref}
\usepackage{mathrsfs}
\usepackage{mathtools}
\usepackage{slashed}

\DeclareMathOperator{\tr}{tr}
\DeclareMathOperator{\Vol}{Vol}
\DeclareMathOperator{\dvol}{dV}

\DeclareMathOperator{\supp}{supp}

\DeclareMathOperator{\Ric}{Ric}

\DeclareMathOperator{\Conf}{Conf}

\DeclareMathOperator{\sech}{sech}

\newcommand{\defn}[1]{{\boldmath\bfseries#1}}

\newcommand{\sff}{I\!I}
\newcommand{\tfsff}{\mathring{\sff}}

\newcommand{\CC}{C_c}

\newcommand{\hg}{\widehat{g}}
\newcommand{\hh}{\widehat{h}}

\newcommand{\hM}{\widehat{M}}

\newcommand{\hSigma}{\widehat{\Sigma}}

\newcommand{\lp}{\langle}
\newcommand{\rp}{\rangle}
\newcommand{\lv}{\lvert}
\newcommand{\rv}{\rvert}
\newcommand{\lV}{\lVert}
\newcommand{\rV}{\rVert}

\newcommand{\mD}{\mathcal{D}}
\newcommand{\mE}{\mathcal{E}}
\newcommand{\mF}{\mathcal{F}}

\newcommand{\mQ}{\mathcal{Q}}

\newcommand{\bN}{\mathbb{N}}

\newcommand{\bR}{\mathbb{R}}

\newcommand{\sD}{\mathscr{D}}

\newcommand{\Yext}{Y^{\mathrm{ext}}}

\def\sideremark#1{\ifvmode\leavevmode\fi\vadjust{\vbox to0pt{\vss
 \hbox to 0pt{\hskip\hsize\hskip1em
 \vbox{\hsize3cm\tiny\raggedright\pretolerance10000
 \noindent #1\hfill}\hss}\vbox to8pt{\vfil}\vss}}}

\newcommand{\suchthat}{\mathrel{}:\mathrel{}}

\newtheorem{theorem}{Theorem}[section]
\newtheorem{proposition}[theorem]{Proposition}
\newtheorem{lemma}[theorem]{Lemma}
\newtheorem{corollary}[theorem]{Corollary}

\theoremstyle{definition}

\theoremstyle{remark}
\newtheorem{remark}[theorem]{Remark}

\numberwithin{equation}{section}

\begin{document}

\title{The sharp $L^2$-Michael--Simon inequality}
\begin{abstract}
    Let $n \geq 3$.
    We prove that if $j \colon (\Sigma^n,g) \to (\mathbb{R}^N,dx^2)$ is an isometric immersion, then
    \begin{equation*}
        \int_\Sigma \left( \lvert \nabla u \rvert^2 + \frac{n(n-2)}{4}\lvert H \rvert^2 u^2 \right) \geq \frac{n(n-2)}{4}\mathrm{Vol}(S^n)^{2/n}\left( \int_\Sigma \lvert u \rvert^{\frac{2n}{n-2}} \right)^{\frac{n-2}{n}}
    \end{equation*}
    for all $u \in W^{1,2}(\Sigma)$, and characterize immersions for which equality is realized.
    Our proof exploits the conformal invariance of the inequality, and hence extends to sharp Michael--Simon inequalities for submanifolds of the sphere, of hyperbolic space, and of other conformally developable manifolds.
    As an application, we derive Chen--Fenchel--Willmore-type inequalities relating $L^p$-norms of the mean curvature vector to the volume of a compact submanifold of a simply-connected spaceform.
\end{abstract}
\author{Jeffrey S.\ Case}
\address{Department of Mathematics \\ Penn State University \\ University Park, PA 16802 \\ USA}
\email{jscase@psu.edu}
\author{Dawit Mengesha}
\address{Department of Mathematics \\ Penn State University \\ University Park, PA 16802 \\ USA}
\email{dym5466@psu.edu}

\maketitle
\section{Introduction}

The classical \defn{$L^2$-Sobolev inequality}~\cites{Aubin1976s,Talenti1976} states that if $n \geq 3$, then
\begin{equation}
    \label{eqn:sobolev}
    \int_{\mathbb{R}^n} \lvert \nabla u \rvert^2 \dvol \geq A_n\left( \int_{\mathbb{R}^n} \lvert u \rvert^{\frac{2n}{n-2}} \dvol \right)^{\frac{n-2}{n}}
\end{equation}
for all $u \in W^{1,2}(\mathbb{R}^n)$, with equality if and only if there are constants $a \in \mathbb{R}$, $b > 0$, and $x_0 \in \mathbb{R}^n$ such that $u(x) = a\bigl( b + \lvert x-x_0 \rvert^2 \bigr)^{-\frac{n-2}{2}}$.
Here all geometric quantities are computed with respect to the Euclidean metric and
\begin{equation}
    \label{eqn:aubin-talenti}
    A_n := \frac{n(n-2)}{4}\Vol(S^n)^{\frac{2}{n}}
\end{equation}
is the \defn{Albin--Talenti constant}, expressed in terms of the intrinsic volume of the unit $n$-sphere.
A fundamental property of Inequality~\eqref{eqn:sobolev} is its conformal invariance.
This fact leads to a proof of the sharp $L^2$-Sobolev inequality that avoids symmetric rearrangement~\cite{FrankLieb2012b} and to a proof of the sharp functional inequality
\begin{equation}
    \label{eqn:sobolev-sphere}
    \int_{S^n} \Bigl( \lvert \nabla u \rvert^2 + \frac{n(n-2)}{4}u^2 \Bigr) \dvol \geq A_n\left( \int_{S^n} \lvert u \rvert^{\frac{2n}{n-2}} \dvol \right)^{\frac{n-2}{n}}
\end{equation}
on the round $n$-sphere $(S^n,d\theta^2)$~\cite{Beckner1993};
it is also fundamental to the resolution of the Yamabe Problem~\cite{LeeParker1987}.
Notably, equality holds for a nonzero function $u$ in Inequality~\eqref{eqn:sobolev-sphere} if and only if $\lvert u \rvert^{\frac{4}{n-2}}d\theta^2$ is a round metric, explaining the geometric significance of the constant $A_n$.

The \defn{$L^2$-Michael--Simon inequality}~\cites{MichaelSimon1973,Allard1972} states that for integers $N \geq n \geq 3$, there is a constant $C_{n,N}>0$ such that if $j \colon (\Sigma^n,g) \to (\mathbb{R}^N,dx^2)$ is an isometric immersion, then
\begin{equation}
    \label{eqn:michael-simon}
    \int_{\Sigma} \Bigl( \lvert \nabla u \rvert^2 + \frac{n(n-2)}{4}\lvert H \rvert^2 u^2 \Bigr) \dvol \geq C_{n,N}\left( \int_{\Sigma} \lvert u \rvert^{\frac{2n}{n-2}} \dvol \right)^{\frac{n-2}{n}}
\end{equation}
for all $u \in W^{1,2}(\Sigma)$.
Here $H := \frac{1}{n}\tr \sff$ is the mean curvature of $j$ and the constant in front of $\lv H \rv^2$ is chosen based on the conformal interpretation of the $L^2$-Michael--Simon inequality described below.
Inequality~\eqref{eqn:michael-simon} follows from an $L^1$-analogue proven by Michael and Simon~\cite{MichaelSimon1973};
see also the work of Bombieri, De Giorgi, and Miranda~\cite{BombieriDeGiorgiMiranda1969} for the case of minimal submanifolds of $\bR^n$.
A crucial point is that $C_{n,N}$ depends only on the dimension and codimension of $j(\Sigma)$, which leads to many applications;
e.g.\ the classification of stable minimal hypersurfaces in $\bR^n$~\cites{SchoenSimonYau1975,ChodoshLi2023,Bellettini2025}.

In analogy with the Sobolev inequality, it is desirable to know the best constant in the $L^2$-Michael--Simon inequality and its geometric significance.
Brendle~\cite{Brendle2021} made a major breakthrough by proving the sharp $L^1$-Michael--Simon inequality when $N \leq n+2$.
In the case of minimal submanifolds in Euclidean space, this estimate implies, using a rearrangement argument and the co-area formula~\cites{Brendle2023,BaloghKristalyMester2026}, that $C_{n,N}=A_n$ when $N \leq n+2$.
Brendle also gave an explicit estimate for $C_{n,N}$, though his estimate tends to zero as $N \to \infty$.
This estimate was improved by Balogh, Krist\'aly, and Mester~\cite{BaloghKristalyMester2026} to a codimension-independent estimate for $C_{n,N}$ on minimal submanifolds;
their estimate satisfies $C_{n,n+m} / A_n \to 1$ as $n \to \infty$ for any fixed codimension $m$.
Han and Xu~\cite{HanXu2026} recently extended Brendle's sharp estimate to dimensions $N \leq n+4$.

Our main result computes the best constant in the $L^2$-Michael--Simon inequality;
it is the Aubin--Talenti constant, regardless of codimension.

\begin{theorem}
    \label{thm:sharp-michael-simon}
    Let $j \colon (\Sigma^n,g) \to (\mathbb{R}^N,dx^2)$, $n \geq 3$, be an isometric immersion.
    Then
    \begin{equation}
        \label{eqn:sharp-michael-simon}
        \int_\Sigma \Bigl( \lvert \nabla u \rvert^2 + \frac{n(n-2)}{4}\lvert H \rvert^2 u^2 \Bigr) \dvol \geq A_n \left( \int_\Sigma \lvert u \rvert^{\frac{2n}{n-2}} \dvol \right)^{\frac{n-2}{n}}
    \end{equation}
    for all $u \in W^{1,2}(\Sigma)$.
    Moreover, equality holds for some nonzero $u \in W^{1,2}(\Sigma)$ if and only if $(\Sigma^n,u^{\frac{4}{n-2}}g)$ is quasi-equal to $(S^n,d\theta^2)$ and $j$ is umbilic.
\end{theorem}

All manifolds are taken to be connected in this paper.
We say that $(\Sigma^n,g_u)$ is \defn{quasi-equal} to $(S^n,d\theta^2)$ if $g_u$ has constant positive sectional curvature and $\Sigma$ is diffeomorphic to an open submanifold $\Sigma' \subset S^n$ whose complement has capacity
\begin{equation*}
    \mathrm{cap}(S^n \setminus \Sigma') := \inf_{u \in W^{1,2}(S^n)} \left\{ \int_{S^n} \Bigl( \lv \nabla u \rv^2 + \frac{n(n-2)}{4}u^2 \Bigr) \dvol \suchthat u \vert_{S^n \setminus \Sigma'} = 1 \right\} = 0 .
\end{equation*}
In particular, if $j$ is proper, then equality holds in Inequality~\eqref{eqn:sharp-michael-simon} for some nonzero $u$ if and only if $(\Sigma,g)$ is conformally equivalent to $(\bR^n,dx^2)$ or $(S^n,d\theta^2)$ and $j(\Sigma)$ is an affine $n$-plane or a geodesic sphere in an affine $(n+1)$-plane, respectively.

Our key insight is that the $L^2$-Michael--Simon inequality admits a conformally invariant formulation to which one can adapt the rearrangement-free argument of Frank and Lieb~\cite{FrankLieb2012b} for proving the sharp $L^2$-Sobolev inequality.
An early hint of this conformal invariance appears in the work of El Soufi and Ilias~\cite{ElSoufiIlias1992} extending Reilly's estimate~\cite{Reilly1977e} of the first nonzero eigenvalue of the Laplacian of a compact submanifold of Euclidean space to the case of submanifolds of conformally developable manifolds;
see Theorem~\ref{thm:elsoufi-ilias} for a precise statement.

Brendle's original estimate~\cite{Brendle2021} of the sharp constant in the $L^1$-Michael--Simon inequality uses the Alexandrov--Bakelman--Pucci method (cf.\ \cite{Cabre2008}), while the later reproof by Brendle and Eichmair~\cite{BrendleEichmair2023} and its improvement by Balogh, Krist\'aly, and Mester~\cite{BaloghKristalyMester2026} uses optimal transport.
Notably, these methods produce $L^p$-inequalities for any $p \in [1,\infty)$ and can be extended to other manifolds under suitable curvature assumptions (e.g.\ \cites{DongLinLu2024,Wang2024,Brendle2023n}).
The difficulty in obtaining the best constant in the $L^2$-Michael--Simon inequality using these methods stems from the fact that they involve studying an auxiliary second-order PDE or optimal transport problem, and the ``best'' choice of problem is not clear.
Moreover, these methods do not extend to higher-order sharp inequalities.
By contrast, conformal invariance removes any choices, and the Frank--Lieb method can be extended to certain higher-order or fully nonlinear inequalities (cf.\ \cite{Case2019fl}).

We now elaborate on the conformal invariance of the $L^2$-Michael--Simon inequality.
Let $j \colon (\Sigma^n,g) \to (M^N,h)$, $N \geq 3$, be an isometric immersion.
The \defn{extrinsic conformal Laplacian}~\cite{CaseGrahamKuo2023} $L \colon C^\infty(\Sigma) \to C^\infty(\Sigma)$ is defined by
\begin{equation}
    \label{eqn:defn-extrinsic-conformal-laplacian}
    \begin{split}
    L & := \Delta + \frac{n-2}{2}Q , \\
    Q & := \tr_g j^\ast P + \frac{n}{2}\lvert H \rvert^2 ,
    \end{split}
\end{equation}
where the Laplacian $\Delta := -\nabla^\alpha\nabla_\alpha$ is nonnegative and
\begin{equation*}
    P := \frac{1}{N-2}\left( \Ric - \frac{R}{2(N-1)}h \right)
\end{equation*}
is the Schouten tensor of $h$.
Direct computation shows that if $(M,h)$ is locally conformally flat or if $N=n+1$, then
\begin{equation}
    \label{eqn:q-to-sectional}
    \frac{2}{n}Q = \lvert H \rvert^2 + \frac{1}{n(n-1)}\sum_{\alpha,\beta=1}^n K(j_\ast\partial_\alpha , j_\ast\partial_\beta)
\end{equation}
at $p \in \Sigma$, where $\{\partial_\alpha \}$ is an orthonormal basis for $T_p\Sigma$ and $K(j_\ast\partial_\alpha , j_\ast\partial_\beta)$ is the sectional curvature with respect to $h$ of the plane spanned by $j_\ast\partial_\alpha$ and $j_\ast\partial_\beta$.
Thus $Q$ is the integrand in the El Soufi--Ilias eigenvalue estimate~\cite{ElSoufiIlias1992}.

The extrinsic conformal Laplacian has many important properties.
It is natural, which implies that if $\Psi \colon (\Sigma,g) \to (\hSigma,\hg)$ and $\Phi \colon (M,h) \to (\hM,\hh)$ are isometries, then the extrinsic conformal Laplacians of $j$ and $\widehat{\jmath} := \Phi \circ j \circ \Psi^{-1}$ are related by
\begin{equation}
    \label{eqn:natural}
    \Psi^\ast ( \widehat{L}\widehat u ) = L(\Psi^\ast\widehat{u}) ;
\end{equation}
see Section~\ref{sec:sharpness} for further details.
It is \defn{formally self-adjoint};
i.e.
\begin{equation*}
    \int_\Sigma uLv \dvol = \int_\Sigma vLu \dvol
\end{equation*}
for all $u,v \in \CC^\infty(\Sigma)$.
It is \defn{conformally invariant};
i.e.\ if $\widehat{h} = e^{2\Upsilon}h$, then
\begin{equation}
    \label{eqn:conformal-transformation}
    \widehat{L}(u) = e^{-\frac{n+2}{2}j^\ast\Upsilon} L\left( e^{\frac{n-2}{2}j^\ast\Upsilon}u \right)
\end{equation}
for all $u \in C^\infty(\Sigma)$.
This determines the conformal transformation law of $Q$ when $n\not=2$;
when $n=2$ one has
\begin{equation}
    \label{eqn:q-transformation}
    e^{2j^\ast\Upsilon}\widehat{Q} = Q + L(\Upsilon) .
\end{equation}
In particular, $\int Q\dvol$ is a conformal invariant of compact immersed surfaces that generalizes the Willmore functional.
Since $L$ and $Q$ are local objects and the transformation laws~\eqref{eqn:conformal-transformation} and~\eqref{eqn:q-transformation} depend only on $j^\ast\Upsilon$, both $L$ and $Q$ are defined whenever $j \colon (\Sigma^n,g) \to (M^N,h)$ is a conformal immersion~\cite{CaseGrahamKuo2023}.

\begin{remark}
    The original motivation~\cite{CaseGrahamKuo2023} for defining the extrinsic conformal Laplacian by Equation~\eqref{eqn:defn-extrinsic-conformal-laplacian} was to ensure that $Q$ is constant for any minimal submanifold of an Einstein manifold.
    Indeed, if $j \colon (\Sigma^n,g) \to (M^N,h)$ is an isometric immersion into an Einstein manifold with $\Ric^h = (N-1)\kappa h$, then
    \begin{equation}
        \label{eqn:q-on-einstein}
        Q = \frac{n}{2} \left( \lv H \rv^2 + \kappa \right) .
    \end{equation}
\end{remark}

Theorem~\ref{thm:sharp-michael-simon} follows from a stronger result about sharp $L^2$-Michael--Simon inequalities for manifolds that are conformally immersed in the round sphere $(S^N,d\theta)$.

\begin{theorem}
    \label{thm:michael-simon-sphere}
    Let $j \colon (\Sigma^n,g) \to (S^N,d\theta^2)$, $n \geq 3$, be a conformal immersion.
    Then
    \begin{equation}
        \label{eqn:michael-simon-sphere}
        \int_\Sigma \Bigl( \lvert \nabla u \rvert^2 + \frac{n-2}{2}Qu^2 \Bigr) \dvol \geq A_n \left( \int_\Sigma \lvert u \rvert^{\frac{2n}{n-2}} \dvol \right)^{\frac{n-2}{n}}
    \end{equation}
    for all $u \in W^{1,2}(\Sigma)$.
    Moreover, equality holds for some nonzero $u \in W^{1,2}(\Sigma)$ if and only if $(\Sigma^n,u^{\frac{4}{n-2}}g)$ is quasi-equal to $(S^n,d\theta^2)$ and, after left-compositing with a conformal diffeomorphism of $(S^N,d\theta^2)$ if necessary, $j$ is a totally geodesic isometric embedding.
\end{theorem}

A manifold $(M^N,h)$ is \defn{conformally developable} if there is conformal immersion $\Phi \colon (M^N,h) \to (S^N,d\theta^2)$;
e.g.\ simply-connected locally conformally flat manifolds are conformally developable~\cite{Kuiper1949}.
Theorem~\ref{thm:michael-simon-sphere} trivially extends to conformal immersions into conformally developable manifolds.
However, it does not extend to conformal immersions into locally conformally flat manifolds;
e.g.\ if $T^n$ is a totally geodesic torus in $T^N$, then $Q=0$ and hence the best constant in Inequality~\eqref{eqn:michael-simon-sphere} is $0$.
A particularly interesting family of conformally developable manifolds are simply-connected spaceforms.

\begin{corollary}
    \label{cor:michael-simon-spaceform}
    Let $j \colon (\Sigma^n,g) \to (M_\kappa^N,h)$, $n \geq 3$, be an isometric immersion into a simply-connected manifold with constant sectional curvature $\kappa$.
    Then
    \begin{equation*}
        \int_\Sigma \Bigl( \lvert \nabla u \rvert^2 + \frac{n(n-2)}{4}\bigl( \lvert H \rvert^2 + \kappa \bigr) u^2 \Bigr) \dvol \geq A_n \left( \int_\Sigma \lvert u \rvert^{\frac{2n}{n-2}} \dvol \right)^{\frac{n-2}{n}}
    \end{equation*}
    for all $u \in W^{1,2}(\Sigma)$.
    Moreover, equality holds for some nonzero $u \in W^{1,2}(\Sigma)$ if and only if $(\Sigma^n,u^{\frac{4}{n-2}}g)$ is quasi-equal to $(S^n,d\theta^2)$ and $j$ is umbilic.
\end{corollary}

Corollary~\ref{cor:michael-simon-spaceform} is new except for submanifolds of $\bR^N$, $N \leq n+4$, and of $S^{N}$, $N \leq n+3$, both of which follow from Han and Xu's result~\cite{HanXu2026}.
In hyperbolic space, Cui and Zhao~\cite{CuiZhao2024} proved an $L^1$-Michael--Simon-type inequality for a restricted class of functions on $h$-convex hypersurfaces; Corollary~\ref{cor:michael-simon-spaceform} significantly relaxes these assumptions, though using instead the $L^2$-norm of the gradient.

It is illuminating to compare Theorem~\ref{thm:michael-simon-sphere} to the El Soufi--Ilias estimate~\cite{ElSoufiIlias1992}.

\begin{theorem}
    \label{thm:elsoufi-ilias}
    Let $(\Sigma^n,g)$, $n \geq 2$, be a compact Riemannian manifold and let $j \colon (\Sigma^n,g) \to (S^N,d\theta^2)$ be a conformal immersion.
    Then
    \begin{equation*}
        \int_\Sigma Q \dvol \geq \frac{\lambda_1}{2}\Vol(\Sigma) ,
    \end{equation*}
    where $\lambda_1$ is the first nonzero eigenvalue of the Laplacian of $g$.
\end{theorem}

The proofs of Theorems~\ref{thm:michael-simon-sphere} and~\ref{thm:elsoufi-ilias} share two applications of conformal invariance;
these are most clearly described in dimensions $n \geq 3$.
First, Equation~\eqref{eqn:conformal-transformation} implies that if $\widehat{g} = e^{2\Upsilon}g$, then
\begin{equation*}
    \int_\Sigma u\widehat{L}u \, \widehat{\dvol} = \int_\Sigma e^{\frac{n-2}{2}\Upsilon}u L\left( e^{\frac{n-2}{2}\Upsilon}u \right) \dvol
\end{equation*}
for all $u \in C^\infty(\Sigma)$.
Second, by left-composing $j$ with an element of the conformal group $\Conf(S^N)$ if necessary, one can assume that the pullbacks $j^\ast x^i$ of first spherical harmonics have mean zero~\cite{LiYau1982}.
In the proof of Theorem~\ref{thm:michael-simon-sphere}, one uses the functions $j^\ast x^i$ as test functions in the second variation to show that the existence of a minimizer imposes strong conditions on the immersion $j$.
This idea is originally due to Hersch~\cite{Hersch1970}, with significant further development by Frank and Lieb~\cite{FrankLieb2012b};
an adaptation of an argument of Aubin~\cite{Aubin1976} allows one to construct a minimizer if the best constant is smaller than $A_n$.
In the proof of Theorem~\ref{thm:elsoufi-ilias}, one directly uses the functions $j^\ast x^i$ as test functions to estimate $\lambda_1$.

Theorem~\ref{thm:michael-simon-sphere} has a number of interesting applications.

First, it specializes to an elegant sharp Sobolev inequality on minimal submanifolds of simply-connected spaceforms.

\begin{corollary}
    \label{cor:minimal-michael-simon}
    Let $j \colon (\Sigma^n,g) \to (M_\kappa^N,h)$, $n \geq 3$, be a minimal isometric immersion into a simply-connected manifold with constant sectional curvature $\kappa$.
    Then
    \begin{equation*}
        \int_\Sigma \Bigl( \lvert \nabla u \rvert^2 + \frac{n(n-2)}{4}\kappa u^2 \Bigr) \dvol \geq A_n \left( \int_\Sigma \lvert u \rvert^{\frac{2n}{n-2}} \dvol \right)^{\frac{n-2}{n}}
    \end{equation*}
    for all $u \in W^{1,2}(\Sigma)$.
    Moreover, equality holds for some nonzero $u \in W^{1,2}(\Sigma)$ if and only if $(\Sigma^n,u^{\frac{4}{n-2}}g)$ is quasi-equal to $(S^n,d\theta^2)$ and $j$ is umbilic.
\end{corollary}

Second, applying the sharp $L^2$-Michael--Simon inequality~\eqref{eqn:michael-simon-sphere} to the constant function recovers sharp Chen--Fenchel--Willmore-type inequalities~\cites{Chen1971,Fenchel1929,Willmore1968};
i.e.\ inequalities relating the $L^p$-norm, $p \in [2,n]$, of the mean curvature and the volume.
The results in the case $p=2$ are the strongest.

\begin{corollary}
    \label{cor:willmore}
    Let $(\Sigma^n,g)$, $n \geq 2$, be a compact Riemannian manifold and let $j \colon (\Sigma^n,g) \to (S^N,d\theta^2)$ be a conformal immersion.
    Then
    \begin{equation}
        \label{eqn:willmore}
        \frac{2}{n}\int_\Sigma Q \dvol \geq \Vol(S^n)^\frac{2}{n}\Vol(\Sigma)^{\frac{n-2}{n}} .
    \end{equation}
    Moreover,
    \begin{itemize}
        \item if $n=2$, then equality holds in Inequality~\eqref{eqn:willmore} if and only if $\Sigma^2$ is diffeomorphic to $S^2$ and $j$ is an umbilic embedding, and
        \item if $n\geq3$, then equality holds in Inequality~\eqref{eqn:willmore} if and only if $(\Sigma^n,g)$ is homothetic to $(S^n,d\theta^2)$ and $j$ is an umbilic isometric embedding.
    \end{itemize}
\end{corollary}

In particular, if $j \colon (\Sigma^n,g) \to (M_\kappa^N,h)$ is an isometric immersion of a compact manifold into a simply-connected spaceform with sectional curvature $\kappa$, then
\begin{equation}
    \label{eqn:willmore-spaceform}
    \int_\Sigma \left( \lv H \rv^2 + \kappa \right) \dvol \geq \Vol(S^n)^\frac{2}{n}\Vol(\Sigma)^{\frac{n-2}{n}}
\end{equation}
with equality if and only if $j$ is umbilic.

Corollary~\ref{cor:willmore} makes no assumptions on either the convexity or the codimension of $j(\Sigma)$, thereby significantly strengthening previous results.
For submanifolds of codimension at most four in Euclidean space and for submanifolds of codimension at most three in the round sphere, Corollary~\ref{cor:willmore} follows from Han and Xu's work~\cite{HanXu2026}.
The best previous result in hyperbolic space, due to Hu~\cite{Hu2018}, used geometric flow methods (cf.\ \cites{LiWeiXiong2014,GuanLi2009}) to prove Inequality~\eqref{eqn:willmore-spaceform} for star-shaped mean convex hypersurfaces.
Corollary~\ref{cor:willmore} also has an interesting interpretation in the context of Santal\'o's conjecture~\cite{Santalo1963}, which posited that
\begin{equation*}
    \int_\Sigma H \, \dvol \geq \sqrt{\Vol(S^2)\Vol(\Sigma) + \Vol(\Sigma)^2}
\end{equation*}
on any compact strictly convex hypersurface in hyperbolic $3$-space.
This would imply Corollary~\ref{cor:willmore}.
However, Santal\'o's conjecture is false~\cites{Natario2015,Hong2026}, making Inequality~\eqref{eqn:willmore-spaceform} all the more remarkable.

While weaker than $L^2$-inequalities, $L^n$-inequalities on the mean curvature have interest by virtue of their scale invariance.

\begin{corollary}
    \label{cor:chen-willmore}
    Let $(\Sigma^n,g)$, $n \geq 2$, be compact and let $j \colon (\Sigma^n,g) \to (S^N,d\theta^2)$ be a conformal immersion.
    Then
    \begin{equation*}
        \int_\Sigma Q_+^{n/2} \dvol \geq \left( \frac{n}{2} \right)^{n/2} \Vol(S^n) .
    \end{equation*}
    Moreover,
    \begin{itemize}
        \item if $n=2$, then equality holds in Inequality~\eqref{eqn:willmore} if and only if $\Sigma^2$ is diffeomorphic to $S^2$ and $j$ is an umbilic embedding, and
        \item if $n\geq3$, then equality holds in Inequality~\eqref{eqn:willmore} if and only if $(\Sigma^n,g)$ is homothetic to $(S^n,d\theta^2)$ and $j$ is an umbilic isometric embedding.
    \end{itemize}
\end{corollary}

Here $Q_+$ denotes the positive part of $Q$.
Corollary~\ref{cor:chen-willmore} implies that if $\Sigma$ is compact and $j \colon (\Sigma^n,g) \to (M_\kappa^N,h)$ is an isometric immersion into a simply-connected spaceform with sectional curvature $\kappa$, then
\begin{equation*}
    \int_\Sigma \bigl( \lvert H \rvert^2 + \kappa \bigr)_+^{n/2} \dvol \geq \Vol(S^n)
\end{equation*}
with equality if and only if $j$ is umbilic.

Corollary~\ref{cor:chen-willmore} is already known for isometric immersions into Euclidean space, and hence also isometric immersions into the round sphere.
The first proof, due to Chen~\cite{Chen1971}, uses properties of the Gauss map~\cite{ChernLashof1957};
see also the work of Heintze and Karcher~\cite{HeintzeKarcher1978}.
PDE methods were used to give alternative proofs for hypersurfaces using a careful study of electrostatic potentials~\cites{AgostinianiMazzieri2020,CederbaumMiehe2025}.
The other cases of Corollary~\ref{cor:chen-willmore} appear new.

We conclude with an additional sense in which the $L^2$-Michael--Simon inequality~\eqref{eqn:michael-simon-sphere} is sharp.
Unlike on the sphere~\cite{Branson1995}, the space of natural, formally self-adjoint, conformally invariant, linear differential operators with leading-order term $\Delta$ on a submanifold is not one-dimensional;
one can add a nonzero multiple of $\lvert \tfsff \rvert^2$ to $L$ without affecting these properties.
Theorem~\ref{thm:michael-simon-sphere} is optimal among all conformally invariant inequalities.

\begin{theorem}
    \label{thm:michael-simon-super-sharp}
    Let $\varepsilon \in \mathbb{R}$ be such that if $j \colon (\Sigma^n,g) \to (S^N,d\theta^2)$, $n \geq 3$, is a conformal immersion, then
    \begin{equation}
        \label{eqn:michael-simon-epsilon}
        \int_\Sigma \left( uLu + \varepsilon\lvert \tfsff \rvert^2 u^2 \right) \dvol \geq A_n \left( \int_\Sigma \lvert u \rvert^{\frac{2n}{n-2}} \dvol \right)^{\frac{n-2}{n}}
    \end{equation}
    for all $u \in \CC^\infty(\Sigma)$.
    Then $\varepsilon \geq 0$.
\end{theorem}

In fact, we show in Section~\ref{sec:sharpness} that if $\varepsilon < 0$, then $L + \varepsilon \lv \tfsff \rv^2 \not\geq 0$ on the product $C^2 \times \bR^{n-2} \subset \bR^{n+1}$ of a catenoid with Euclidean space.

This paper is organized as follows.

In Section~\ref{sec:setup}, we introduce the extrinsic Yamabe constant, which is a conformal invariant that computes the best constant in the $L^2$-Michael--Simon inequality.
We also prove that the extrinsic Yamabe constant is at most $A_n$, and that if the inequality is strict, then a minimizer exists.

In Section~\ref{sec:rigidity}, we carry out the extrinsic version of the Frank--Lieb argument for the extrinsic Yamabe constant.
We then apply this to prove Theorem~\ref{thm:michael-simon-sphere}.

In Section~\ref{sec:applications}, we explain how Theorem~\ref{thm:sharp-michael-simon} and Corollaries~\ref{cor:michael-simon-spaceform}, \ref{cor:minimal-michael-simon}, \ref{cor:willmore}, and~\ref{cor:chen-willmore} follow from Theorem~\ref{thm:michael-simon-sphere}.
We also streamline the proof of Theorem~\ref{thm:elsoufi-ilias} using the extrinsic conformal Laplacian.

In Section~\ref{sec:sharpness}, we prove Theorem~\ref{thm:michael-simon-super-sharp}.

\section{The extrinsic conformal Laplacian}
\label{sec:setup}

In this section we study the extrinsic conformal Laplacian and its associated Yamabe-type problem.
Our treatment takes elements from both the standard approach to the Yamabe Problem~\cite{LeeParker1987} and to the proof of existence via Lions' Concentration Compactness Principle~\cites{Lions1985a}.

Case, Graham, and Kuo~\cite{CaseGrahamKuo2023}*{Theorems~1.1 and~5.2} proved the conformal invariance~\eqref{eqn:conformal-transformation} of the extrinsic conformal Laplacian $L$ by exhibiting it as a member of a family of conformally invariant operators.
A more direct proof, which applies only for the second-order operator $L$, uses the Gauss equations~\cite{ElSoufiIlias1992}*{Proposition~2}.
We recall this argument in order to clarify the relationship between $L$, the usual (intrinsic) conformal Laplacian
\begin{equation*}
    L_{\mathrm{int}} := \Delta + \frac{n-2}{4(n-1)}R ,
\end{equation*}
where $R$ is the scalar curvature, and the extrinsically-coupled conformal Laplacian used by El Soufi and Ilias~\cite{ElSoufiIlias1992}.

\begin{lemma}
    \label{lem:conformal-invariance}
    Let $j \colon (\Sigma^n,g) \to (M^N,h)$, $n \geq 1$ and $N \geq 3$, be a conformal immersion.
    Set $\widehat{g} = e^{2\Upsilon}g$.
    Then
    \begin{equation*}
        L^{\widehat{g}}u = e^{-\frac{n+2}{2}\Upsilon}L^g\bigl( e^{\frac{n-2}{2}\Upsilon}u \bigr)
    \end{equation*}
    for all $u \in C^\infty(\Sigma)$.
\end{lemma}

\begin{proof}
    Since $L$ is a local operator and $g = e^{2u}j^\ast h$ for some $u \in C^\infty(\Sigma)$, we may assume that $j$ is an isometric embedding, and hence that $\Upsilon$ is the pullback of a smooth function on $M$.
    
    The Gauss equation implies that
    \begin{equation*}
        Q = \frac{1}{2(n-1)}\left( R + \lv \tfsff \rv^2 - \tr^2 j^\ast W \right) ,
    \end{equation*}
    where
    \begin{equation*}
        \tr^2 j^\ast W := \sum_{\alpha,\beta=1}^n W(e_\alpha,e_\beta,e_\alpha,e_\beta)
    \end{equation*}
    is the trace of the pullback of the Weyl tensor of $h$;
    this follows, for example, by taking the trace of the equation~\cite{CaseGrahamKuoTyrrellWaldron2023a}*{Equation~(4.9b)} relating the intrinsic Schouten tensor $P^{g}$ to the pullback $j^\ast P^h$ of the Schouten tensor of $h$.
    Therefore
    \begin{equation*}
        L = L_{\mathrm{int}} + \frac{n-2}{4(n-1)}\left( \lv \tfsff \rv^2 - \tr^2 j^\ast W \right) .
    \end{equation*}
    The conclusion follows from the conformal invariance of the conformal Laplacian~\cite{LeeParker1987} and the conformal invariance of the trace-free second fundamental form $\tfsff$ and the Weyl tensor $W$.
\end{proof}

Four key properties of the extrinsic conformal Laplacian are that it is natural, conformally invariant, formally self-adjoint, and has leading-order term $\Delta$.
The last three properties are equivalently formulated using the associated quadratic form
\begin{equation*}
    \mQ(u,v) := \int_\Sigma uLv \dvol ,
\end{equation*}
defined on $\CC^\infty(\Sigma) \times \CC^\infty(\Sigma)$.

\begin{lemma}
    \label{lem:energy-invariance}
    Let $j \colon (\Sigma^n,g) \to (M^N,h)$, $n \geq 1$ and $N \geq 3$, be a conformal immersion.
    Then
    \begin{equation}
        \label{eqn:energy-integral}
        \mQ(u,v) = \int_\Sigma \Bigl( \lp \nabla u , \nabla v \rp + \frac{n-2}{2}Quv \Bigr) \dvol
    \end{equation}
    for all $u,v \in \CC^\infty(\Sigma)$.
    In particular, $\mQ$ is symmetric on $\CC^\infty(\Sigma) \times \CC^\infty(\Sigma)$.
    Moreover, if $\Upsilon \in C^\infty(\Sigma)$, then
    \begin{equation}
        \label{eqn:conformally-covariant}
        \mQ^{e^{2\Upsilon}g}(u,v) = \mQ^g\bigl( e^{\frac{n-2}{2}\Upsilon}u , e^{\frac{n-2}{2}\Upsilon}v \bigr)
    \end{equation}
    for all $u,v \in \CC^\infty(\Sigma)$.
\end{lemma}

\begin{proof}
    Direct computation yields Equation~\eqref{eqn:energy-integral}, which is evidently symmetric.

    Set $\hg := e^{2\Upsilon}g$.
    Lemma~\ref{lem:conformal-invariance} implies that
    \begin{equation*}
        \mQ^{\hg}(u,v) = \int_\Sigma e^{\frac{n-2}{2}\Upsilon}u L^g\bigl( e^{\frac{n-2}{2}\Upsilon}v \bigr) \dvol_g = \mQ^g\bigl( e^{\frac{n-2}{2}\Upsilon}u , e^{\frac{n-2}{2}\Upsilon}v \bigr) . \qedhere
    \end{equation*}
\end{proof}

We now introduce the extrinsic Yamabe problem.
Because the $L^2$-Michael--Simon inequality applies to noncompact submanifolds, and because functional inequalities on noncompact manifolds are propitiously studied using compact exhaustions, we find it advantageous to define the extrinsic Yamabe constant on (possibly noncompact) Riemannian manifolds-with-boundary.

Let $(\Sigma^n,g)$, $n \geq 3$, be a Riemannian manifold-with-boundary.
The \defn{extrinsic Yamabe energy} $\mE$ of a conformal immersion $j \colon (\Sigma^n,g) \to (M^N,h)$ is defined on $\CC^\infty(\Sigma \setminus \partial\Sigma)$ by
\begin{equation*}
    \mE(u) := \mQ(u,u) .
\end{equation*}
Lemma~\ref{lem:energy-invariance} implies that the extrinsic Yamabe energy is conformally invariant and
\begin{equation*}
    \frac{1}{2} \left. \frac{d}{dt} \right|_{t=0} \mE(u + tv) = \mQ(u,v)
\end{equation*}
for any $u,v \in \CC^\infty(\Sigma \setminus \partial\Sigma)$.
The \defn{extrinsic Yamabe constant} of $j$ is
\begin{equation}
    \label{eqn:defn-extrinsic-yamabe-constant}
    \Yext(j) := \inf \left\{ \mE(u) \suchthat u \in \CC^\infty( \Sigma \setminus \partial\Sigma ) , \int_\Sigma \lv u \rv^{\frac{2n}{n-2}} \dvol = 1 \right\} .
\end{equation}
Equivalently,
\begin{align*}
    \Yext(j) & = \inf \left\{ \mF(u) \suchthat u \in \CC^\infty(\Sigma \setminus \partial\Sigma) \right\} , \\
    \mF(u) & := \mE(u) \left( \int_\Sigma \lv u \rv^{\frac{2n}{n-2}} \dvol \right)^{-\frac{n-2}{n}} .
\end{align*}
That the extrinsic Yamabe constant is well-defined and natural follows from the conformal invariance and naturality of the extrinsic Yamabe energy.

\begin{lemma}
    \label{lem:yamabe-constant-conformal-invariant}
    Let $(\Sigma^n,g)$, $n \geq 3$, be a Riemannian manifold-with-boundary and let $j \colon (\Sigma^n,g) \to (M^N,h)$ be a conformal immersion.
    Then $\Yext(j)$ is well-defined.
    Moreover, if $\Phi \colon (M,h) \to (\hM,\hh)$ is a conformal diffeomorphism, then
    \begin{equation*}
        \Yext(j) = \Yext(\Phi \circ j) .
    \end{equation*}
\end{lemma}

\begin{proof}
    To show that $\Yext(j)$ is well-defined, we must show that if $\hg = e^{2\Upsilon}g$, then
    \begin{multline*}
        \inf \left\{ \mE^{\hg}(u) \suchthat u \in \CC^\infty(\Sigma \setminus \partial\Sigma) , \int_\Sigma \lv u \rv^{\frac{2n}{n-2}} \dvol_{\hg} = 1 \right\} \\
            = \inf \left\{ \mE^g(u) \suchthat u \in \CC^\infty(\Sigma \setminus \partial\Sigma) , \int_\Sigma \lv u \rv^{\frac{2n}{n-2}} \dvol_g = 1 \right\} .
    \end{multline*}
    Let $u \in \CC^\infty(\Sigma \setminus \partial\Sigma)$.
    On the one hand, Lemma~\ref{lem:energy-invariance} implies that
    \begin{equation}
        \label{eqn:energy-invariance}
        \mE^{\hg}(u) = \mE^g\bigl( e^{\frac{n-2}{2}\Upsilon}u \bigr) .
    \end{equation}
    On the other hand, direct computation yields
    \begin{equation*}
        \int_\Sigma \lv u \rv^{\frac{2n}{n-2}} \dvol_{\hg} = \int_\Sigma \lv e^{\frac{n-2}{2}\Upsilon}u \rv^{\frac{2n}{n-2}} \dvol_g .
    \end{equation*}
    It readily follows that $\Yext(j)$ is well-defined.
    The final conclusion readily follows from Equation~\eqref{eqn:natural}.
\end{proof}

The remaining goal of this section is to find a sufficient condition for the existence of a smooth positive minimizer of the extrinsic Yamabe constant~\eqref{eqn:defn-extrinsic-yamabe-constant}.
Analogous to a result of Aubin~\cite{Aubin1976}, our sufficient condition involves a comparison of the extrinsic Yamabe constant and the Aubin--Talenti constant~\eqref{eqn:aubin-talenti}.
We first observe, also in analogy with the Yamabe Problem, that the extrinsic Yamabe constant is always bounded above by the Aubin--Talenti constant $A_n$.

\begin{lemma}
    \label{lemma:aubin-estimate}
    Let $(\Sigma^n,g)$, $n \geq 3$, be a Riemannian manifold-with-boundary and let $j \colon (\Sigma^n,g) \to (M^N,h)$ be a conformal immersion.
    Then
    \begin{equation*}
        \Yext(j) \leq A_n .
    \end{equation*}
\end{lemma}

\begin{proof}
    Pick a point $p \in \Sigma \setminus \partial\Sigma$.
    Since $\Yext(j)$ is conformally invariant, we may assume that $j$ is an isometric immersion.
    By homothetically rescaling $g$ and $h$, we may also assume that the restriction $\exp_p \vert_{B_1} \colon B_1 \to \Sigma$ of the exponential map to the unit ball $B_1 \subset T_p\Sigma$ is a diffeomorphism and that $\Omega := \exp_p(B_1)$ has compact closure in $\Sigma \setminus \partial\Sigma$.
    
    Clearly
    \begin{equation*}
        \Yext( j \vert_\Omega ) \geq \Yext(j) .
    \end{equation*}
    Given $\varepsilon>0$, define $\delta_\varepsilon \colon T_pM \to T_pM$ by $\delta_\varepsilon(X) := \varepsilon X$.
    Consider the diffeomorphism
    \begin{equation*}
        \varphi_\varepsilon := \exp_p \mathop{\circ} \delta_\varepsilon \colon B_{\varepsilon^{-1}} \to \Omega .
    \end{equation*}
    On the one hand, Lemma~\ref{lem:yamabe-constant-conformal-invariant} yields
    \begin{equation*}
        \Yext( j \vert_\Omega ) = \Yext( j \circ \varphi_\varepsilon ) .
    \end{equation*}
    On the other hand, direct computation implies that $\varepsilon^{-2}\varphi_\varepsilon^\ast g \to dx^2$ in the compact-open $C^2$-topology as $\varepsilon \to 0$;
    i.e.\ for any compact set $K \subseteq T_p\Sigma$, it holds that
    \begin{equation*}
        \lV \varepsilon^{-2}\varphi_\varepsilon^\ast g - dx^2 \rV_{C^2(K)} \to 0 \quad \text{as $\varepsilon \to 0$}.
    \end{equation*}
    The conclusion follows from the upper semicontinuity of $\Yext$ (cf.\ \cite{GrosseNardmann2014}*{Lemma~4.1}).
\end{proof}

Our goal is to show that on conformal immersions of compact Riemannian manifolds-with-boundary $\Sigma$, if $\Yext(j) < A_n$, then a minimizer of the extrinsic Yamabe constant exists.
To that end, recall that the completion of $C^\infty(\Sigma \setminus \partial\Sigma)$ with respect to the $W^{1,2}$-norm is the space $W_0^{1,2}(\Sigma)$ of $W^{1,2}$-functions with trace zero along $\partial\Sigma$.
The strict inequality prohibits minimizing sequences from concentrating, which is proven differently in two separate cases.

When the extrinsic Yamabe constant is sufficiently small, existence of a minimizer follows from direct arguments (cf.\ \cite{Trudinger1968}*{Theorem~2}).
We state this result for negative extrinsic Yamabe constant, as the nonnegative case can be handled by other methods.

\begin{lemma}
    \label{lem:existence-nonpositive}
    Let $(\Sigma^n,g)$, $n \geq 3$, be a compact Riemannian manifold-with-boundary and let $j \colon (\Sigma^n,g) \to (M^N,h)$ be a conformal immersion such that
    \begin{equation*}
        \Yext(j) < 0 .
    \end{equation*}
    Then there is a positive $u \in C^\infty(\Sigma)$ such that $u \vert_{\partial\Sigma}=0$ and $\int u^{\frac{2n}{n-2}}\dvol = 1$ and
    \begin{equation*}
        \mE(u) = \Yext(j) .
    \end{equation*}
\end{lemma}

\begin{remark}
    Equation~\eqref{eqn:q-on-einstein} implies that the extrinsic Yamabe constant of a conformal immersion of a compact manifold-with-boundary into an Einstein manifold with nonnegative scalar curvature is nonnegative.
    In particular, Lemma~\ref{lem:existence-nonpositive} is not needed in this paper, but rather is included for completeness.
\end{remark}

\begin{proof}
    Define $\mathcal{D} \colon \CC^\infty(\Sigma \setminus \partial\Sigma) \to \bR$ by
    \begin{equation*}
        \mathcal{D}(u) := \mE(u) - \Yext(j) \lV u \rV_{2^\ast}^2 ,
    \end{equation*}
    where $2^\ast := \frac{2n}{n-2}$ and $\lV \cdot \rV_p$ denotes the $L^p$-norm.
    Consider the Riemannian invariant
    \begin{equation*}
        \mathcal{Y}(j,g) := \inf \left\{ \mD(u) \suchthat u \in \CC^\infty(\Sigma \setminus \partial\Sigma) , \lV u \rV_2 = 1 \right\} .
    \end{equation*}
    Clearly $\mathcal{Y}(j,g) \geq 0$.
    Let $(u_k)_{k=1}^\infty$ be a minimizing sequence for the extrinsic Yamabe constant~\eqref{eqn:defn-extrinsic-yamabe-constant}.
    On the one hand, since $\Yext(j) < 0$, it holds that $\mE(u_k) < 0$ for $k \gg 1$.
    On the other hand, the Sobolev embedding $W^{1,2}(\Sigma) \hookrightarrow L^{2^\ast}(\Sigma)$ yields a constant $C>0$ such that
    \begin{equation*}
        C^{-1} \lV v \rV_{2^\ast}^2 \leq \int_\Sigma \bigl( \lv \nabla v \rv^2 + v^2 \bigr) \dvol \leq \mE(v) + C \lV v \rV_2^2
    \end{equation*}
    for all $v \in W_0^{1,2}(\Sigma)$.
    Combining these observations yields $\lV u_k \rV_2^2 \geq C^{-2}$ for all $k \gg 1$.
    Set $\lambda_k := \lV u_k \rV_2^{-1}$.
    Then $\lV \lambda_k u_k \rV_2 = 1$ and $\mathcal{D}(\lambda_ku_k) \to 0$ as $k \to \infty$.
    Therefore $\mathcal{Y}(j,g)=0$.

    Now let $(u_k)$ be a minimizing sequence for $\mathcal{Y}(j,g)$.
    Standard arguments imply that there is a $u \in W_0^{1,2}(\Sigma)$ such that $u_k$ converges weakly in $W_0^{1,2}(\Sigma)$ and $L^{2^\ast}(\Sigma)$, and strongly in $L^2(\Sigma)$, to $u$.
    In particular, $\lV u \rV_2 = 1$.
    Since $\Yext(j) \leq 0$, Fatou's Lemma yields $\mathcal{D}(u) \leq \mathcal{Y}(j,g)=0$.
    Thus $\mathcal{D}(u)=0$, and hence, after rescaling $u$ if necessary, $u$ is a minimizer for $\Yext(j)$.
    The conclusion follows from elliptic regularity and the strong maximum principle (cf.\ \cite{Trudinger1968}*{Theorem~3}).
\end{proof}

When the extrinsic Yamabe constant is nonnegative, proving the existence of a minimizer involves ruling out the possibility that minimizing sequences form bubbles.
We accomplish this using Lions' Concentration Compactness Principle~\cite{Lions1985a}, whose formulation we record in terms of the Dirichlet form~\eqref{eqn:energy-integral}.

\begin{lemma}
    \label{lem:concentration-compactness}
    Let $(\Sigma^n,g)$, $n \geq 3$, be a compact Riemannian manifold-with-boundary and let $j \colon (\Sigma^n,g) \to (S^N,d\theta^2)$ be an isometric immersion.
    Suppose that $(u_k)_{k=1}^\infty$ is a sequence in $\CC^\infty(\Sigma \setminus \partial\Sigma)$ that is bounded in $W_0^{1,2}(\Sigma)$.
    Then there exist $u \in W_0^{1,2}(\Sigma)$ and measures $\mu,\nu$ on $\Sigma$ such that
    \begin{align*}
        u_k & \rightharpoonup u , && \text{in $W_0^{1,2}$}, \\
        \mu_k := \left( \lv \nabla u_k \rv_g^2 + \frac{n-2}{2}Q^gu_k^2 \right) \dvol_g & \to \mu , && \text{as measures} , \\
        \nu_k := \lv u_k \rv^{\frac{2n}{n-2}} \dvol_g & \to \nu , && \text{as measures} .
    \end{align*}
    Additionally, there is a countable set $J \subset \Sigma$ and a set $\{ \nu_p \}_{p \in J} \subset (0,\infty)$ such that
    \begin{equation*}
        \nu = \lv u \rv^{\frac{2n}{n-2}} \dvol_g + \sum_{p \in J} \nu_p \delta_p ,
    \end{equation*}
    where $\delta_p$ denotes the Dirac measure at $p$.
    Moreover,
    \begin{equation*}
        \mu \geq \Bigl( \lv \nabla u \rv_g^2 + \frac{n-2}{2}Qu_k^2 \Bigr) \dvol_g + A_n \sum_{p \in J} \nu_p^{(n-2)/n} \delta_p .
    \end{equation*}
\end{lemma}

\begin{proof}
    Since $\Sigma$ is compact, by picking a subsequence of $(u_k)$ if necessary, we can assume that $(u_k)$ converges weakly in $W_0^{1,2}$ and that $(\mu_k)$ and $(\nu_k)$ converge as measures.
    The conclusion then follows from the Concentration Compactness Principle (cf.\ \cite{Lions1985a}*{Lemma~I.1 and Example~6}).
\end{proof}

We now give our sufficient condition for the existence of a minimizer of $\Yext(j)$.

\begin{proposition}
    \label{prop:aubin-existence}
    Let $(\Sigma^n,g)$, $n \geq 3$, be a compact Riemannian manifold-with-boundary and let $j \colon (\Sigma^n,g) \to (M^N,h)$ be a conformal immersion such that
    \begin{equation}
        \label{eqn:aubin-comparison}
        \Yext(j) < A_n .
    \end{equation}
    Then there is a positive $u \in C^\infty(\Sigma)$ such that $u \vert_{\partial\Sigma}=0$ and $\int u^{\frac{2n}{n-2}} \dvol = 1$ and
    \begin{equation*}
        \mathcal{E}(u) = \Yext(j) .
    \end{equation*}
\end{proposition}

\begin{proof}
    By Lemma~\ref{lem:existence-nonpositive}, it suffices to assume that $\Yext(j) \geq 0$.
    Let $(u_k)_{k=1}^\infty$ be a minimizing sequence for $\Yext(j)$;
    i.e.
    \begin{align*}
        \lV u_k \rV_{2^\ast} & = 1 , && \text{for all $k \in \bN$} , \\
        \mE(u_k) & \to \Yext(j) , && \text{as $k \to \infty$} .
    \end{align*}
    Since $\mE(\lv u_k \rv) \leq \mE(u_k)$, we may also assume that each $u_k$ is nonnegative.
    Observe that $(u_k)$ is bounded in $W_0^{1,2}(\Sigma)$.
    Let $u,\mu,\nu,J,\{ \nu_p \}$ be as in Lemma~\ref{lem:concentration-compactness}.
    Then
    \begin{align}
        \label{eqn:volume-consequence} 1 & = \nu(\Sigma) = \lV u \rV_{2^\ast}^{2^\ast} + \sum_{p \in J} \nu_p , \\
        \label{eqn:energy-consequence} \Yext(j) & = \mu(\Sigma) \geq \mE(u) + A_n\sum_{p \in J} \nu_p^{(n-2)/n} .
    \end{align}
    In particular, $\lV u \rV_{2^\ast} \in [0,1]$ and $\nu_p \in (0,1]$ for all $p \in J$.
    Combining Assumption~\eqref{eqn:aubin-comparison} with Inequality~\eqref{eqn:energy-consequence} yields
    \begin{equation*}
        \Yext(j) \geq \Yext(j) \Bigl( \lV u \rV_{2^\ast}^2 + \sum_{p \in J} \nu_p^{(n-2)/n} \Bigr)
    \end{equation*}
    with equality only if $J = \emptyset$ and $\mE(u) = \Yext(j) \lV u \rV_{2^\ast}^2$.
    Equation~\eqref{eqn:volume-consequence} and the concavity of the function $x \mapsto x^{(n-2)/n}$ on $[0,\infty)$ yield
    \begin{equation*}
        \Yext(j) \Bigl( \lV u \rV_{2^\ast}^2 + \sum_{p \in J} \nu_p^{(n-2)/n} \Bigr) \geq \Yext(j) ,
    \end{equation*}
    and hence $J = \emptyset$ and $u$ is a minimizer of $\Yext(j)$.
    The conclusion follows from elliptic regularity and the strong maximum principle (cf.\ \cite{Trudinger1968}*{Theorem~3}).
\end{proof}

\section{The Frank--Lieb argument}
\label{sec:rigidity}

In this section we adapt the Frank--Lieb argument~\cites{FrankLieb2012b,Case2019fl}, which is a rearrange\-ment-free method to prove sharp Sobolev inequalities, to the extrinsic conformal Laplacian of submanifolds of the sphere.
As an application, we prove Theorem~\ref{thm:michael-simon-sphere}.

The basic idea underlying the Frank--Lieb argument is that if $u$ is a minimizer of the extrinsic Yamabe constant~\eqref{eqn:defn-extrinsic-yamabe-constant} of a conformal immersion into the round $N$-sphere, then summing the second variations in the directions $x^iu$ gives rise to a reverse Poincar\'e inequality from which one can deduce rigidity.
This argument is significantly simplified by using the action of the conformal group of $S^N$ to assume that the conformal immersion is balanced in the sense of Li and Yau~\cite{LiYau1982}, thereby breaking conformal invariance.

\begin{lemma}[\cite{Urbano2025}*{Proposition~1.8}]
    \label{lem:balance}
    Let $j \colon (\Sigma^n,g) \to (S^N,d\theta^2)$ be a conformal immersion of a finite-volume Riemannian manifold-with-boundary into the round $N$-sphere such that $j^{-1}(\{p\})$ has measure zero for each $p \in S^N$.
    Given a positive function $f \in C^\infty(\Sigma)$, there is a $\Phi \in \Conf(S^N)$ such that
    \begin{equation*}
        \int_\Sigma (j^\ast \Phi^\ast x^i) f \dvol_g = 0
    \end{equation*}
    for all $i \in \{ 0 , 1 , \dotsc , N \}$, where $x^i$ are the restriction of the Cartesian coordinates to the unit sphere $S^N \subset \mathbb{R}^{N+1}$.
\end{lemma}

An important step in deducing an effective reverse Poincar\'e inequality from the second variation argument is to compute the weighted sum $\sum j^\ast x^i[L,j^\ast x^i]$ of the commutators of the extrinsic conformal Laplacian and pullbacks of the standard Cartesian coordinates on $S^N$, regarded as multiplication operators.
Since $L$ is second order, this is a simple computation (cf.\ \cite{Case2019fl}*{Theorem~1.3}).

\begin{lemma}
    \label{lem:commutator}
    Let $j \colon (\Sigma^n,g) \to (S^N,d\theta^2)$ be an isometric immersion.
    Then
    \begin{equation*}
        \sum_{i=0}^N j^\ast x^i [ L , j^\ast x^i ] = n .
    \end{equation*}
\end{lemma}

\begin{proof}
    Recall that $\Delta^g j^\ast x^i = nj^\ast x^i + nH(dj^\ast x^i)$, where $H(dj^\ast x^i)$ denotes the canonical pairing of the mean curvature vector $H$ and the one-form $dj^\ast x^i$.
    It follows that
    \begin{equation*}
        [L,j^\ast x^i]u = [\Delta^g,j^\ast x^i]u = nuj^\ast x^i + nuH(dj^\ast x^i) - 2\lp \nabla^g j^\ast x^i , \nabla^g u \rp .
    \end{equation*}
    The conclusion now follows from the identity $\sum (j^\ast x^i)^2 = 1$.
\end{proof}

In order to deduce rigidity, we also require a classification of certain minimal submanifolds of $(S^N,d\theta^2)$ with volume equal to that of the round $n$-sphere.

\begin{lemma}
    \label{lem:minimal-volume}
    Let $j \colon (\Sigma^n,g) \to (S^N,d\theta^2)$ be a minimal isometric immersion of a Riemannian manifold-with-boundary into the round $n$-sphere such that $1 \in W_0^{1,2}(\Sigma)$ and $\Vol(\Sigma) \leq \Vol(S^n)$.
    Then $j$ is totally geodesic and $(\Sigma^n,g)$ is quasi-equal to $(S^n,d\theta^2)$.
\end{lemma}

\begin{remark}
    This recovers the well-known fact that the only compact minimal submanifold $\Sigma^n$ of $(S^N,d\theta^2)$ with $\Vol(\Sigma^n) \leq \Vol(S^n)$ is the totally geodesic $n$-sphere $S^n \subset S^N$~\citelist{ \cite{Brendle2012}*{Corollary~3} \cite{ChengLiYau1984}*{Corollary~5} \cite{Chen1971}*{Theorem~3} }.
\end{remark}

\begin{proof}
    The proof uses geometric measure theory;
    we follow Simon~\cite{Simon83} for notation.

    As an isometric immersion of a finite-volume manifold, $j$ determines an $n$-varifold $[\Sigma]$ supported on $\overline{j(\Sigma)}$.
    Pick a sequence $(\varphi_k)_{k\in\bN}$ in $C_c^\infty(\Sigma)$ that converges to $1$ in $W^{1,2}(\Sigma)$.
    Our first goal is to show that $[\Sigma]$ is stationary.

    Denote by $\delta[\Sigma](S)$ the first variation~\cite{Simon83}*{Equation~39.2} with respect to a smooth vector field $X$ tangent to $S^N$.
    Direct computation then yields
    \begin{align*}
        \biggl| \delta[\Sigma](X) -\delta[\Sigma](\varphi_k X) \biggr| & = \biggl| \int_{\Sigma} \text{div}_{\Sigma}(X) - \text{div}_{\Sigma}(\varphi_k X) d\mathcal{H}^n \biggr|\\
            & \leq \bigg|\int_{\Sigma} (1-\varphi_k) \text{div}_{\Sigma}(X)d\mathcal{H}^{n}\bigg| + \biggl|\int_{\Sigma} \nabla^{\Sigma} \varphi_k\cdot X d\mathcal{H}^n\biggr| .
    \end{align*}
    Since $\varphi_k \to 1$ in $W^{1,2}(\Sigma)$, we deduce that
    \begin{equation*}
        \label{eqn:stationarity}
         \delta[\Sigma](X) = \lim_{k\rightarrow \infty}\delta[\Sigma](\varphi_k X) .
    \end{equation*}
    For each $k$, we may pick an open set $U_k \subset \Sigma$ such that $\supp \varphi_k \subset U_k$.
    Since $\Sigma$ is minimal, we readily deduce that
    \begin{equation*}
        \bigl| \delta[\Sigma](X) \bigr| \leq \lim_{k\rightarrow \infty} \bigl| \delta [\Sigma] (\varphi_kX) \bigr| =\lim_{k\rightarrow \infty} \bigl| \delta[\Sigma] \llcorner  U_k (\varphi_k X) \bigr| = 0 .
    \end{equation*}
    Therefore $[\Sigma]$ is a stationary.

    Now, since $[\Sigma]$ is stationary in $S^N$, the cone $C[\Sigma]$ over $[\Sigma]$ is a stationary $(n+1)$-varifold in $\bR^{N+1}$.
    Since $\Vol(\Sigma^n) \leq \Vol(S^n)$, we deduce from the Monotonicity Formula~\cite{Allard1972} that $C[\Sigma] \llcorner B^{N+1} = [B^{n+1}]$, and hence $[\Sigma] = [S^n]$ and $j$ is an embedding.
    Since $1 \in W_0^{1,2}(\Sigma)$, we readily conclude that $S^n \setminus \Sigma$ has capacity zero.
\end{proof}

We now carry out the Frank--Lieb argument.

\begin{proposition}
    \label{prop:frank-lieb}
    Let $j \colon (\Sigma^n,g) \to (S^N,d\theta^2)$ be a conformal immersion of a Riemannian manifold-with-boundary into the round $N$-sphere.
    Suppose that $u$ is a minimizer of $\Yext(j)$.
    Then $(\Sigma^n,g_u := u^{\frac{4}{n-2}}g)$ is quasi-equal to $(S^n,d\theta^2)$ and, after left-composing with an element of $\Conf(S^N)$ if necessary, $j \colon (\Sigma^n,g_u) \to (S^N,d\theta^2)$ is a totally geodesic isometric embedding.
    In particular,
    \begin{equation*}
        \Yext(j) = A_n .
    \end{equation*}
\end{proposition}

\begin{proof}
    We may assume without loss of generality that $u>0$.
    By Lemma~\ref{lem:balance}, we may assume, by left-composing with an element of $\Conf(S^N)$ if necessary, that
    \begin{align*}
        \int_{\Sigma} (j^\ast x^i) u^{\frac{2n}{n-2}} \dvol_{g} =0 
    \end{align*}
    for all $i \in \{ 0, \dotsc , m+n \}$.
    By conformal invariance, we may also assume $g = j^\ast d\theta^2$.
    
    Since $u$ minimizes the extrinsic Yamabe constant, we compute that
    \begin{equation*}
        0 \leq \frac{1}{2} \sum_{i=0}^N \left. \frac{d^2}{dt^2} \right|_{t=0} \mF\bigl((1+tx^i)u\bigr) = \sum_{i=0}^N \int_\Sigma ux^i L(ux^i) \dvol - \frac{n+2}{n-2} \int_M uLu \dvol .
    \end{equation*}
    Combining this with Lemma~\ref{lem:commutator} yields
    \begin{equation*}
        0 \geq \int_\Sigma u \Bigl( Lu - \frac{n(n-2)}{4}u \Bigr) \dvol = \int_\Sigma \Bigl( \lv \nabla u \rv^2 + \frac{n(n-2)}{4}\lv H \rv^2u^2 \Bigr) \dvol ,
    \end{equation*}
    where the equality uses Lemma~\ref{lem:energy-invariance} and the formula~\eqref{eqn:q-on-einstein} for the $Q$-curvature of an isometrically immersed submanifold of $(S^N,d\theta^2)$.
    It follows that $u$ is constant and $H=0$.
    Combining this with Lemma~\ref{lemma:aubin-estimate} yields
    \begin{equation*}
        \frac{n(n-2)}{4}\Vol(\Sigma)^{2/n} = \Yext(j) \leq A_n = \frac{n(n-2)}{4}\Vol(S^n)^{2/n} .
    \end{equation*}
    In particular, $\Vol(\Sigma) \leq \Vol(S^n)$.
    Lemma~\ref{lem:minimal-volume} implies that $j$ is totally geodesic and that $(\Sigma^n,g)$ is quasi-equal to $(S^n,d\theta^2)$.
    The conclusion follows by unwinding the conformal rescaling from $g$ to $j^\ast d\theta^2$.
\end{proof}

\begin{proof}[Proof of Theorem~\ref{thm:michael-simon-sphere}]
    Suppose first that $\Yext(j) < A_n$.
    Then there is a compact manifold-with-boundary $\Omega_i^n \subset \Sigma$ such that $Y(j \vert_{\Omega}) < A_n$.
    Proposition~\ref{prop:aubin-existence} implies that there is a nonnegative minimizer of $\Yext(j \vert_{\Omega_i} )$.
    We then deduce from Proposition~\ref{prop:frank-lieb} that $\Yext(j) = A_n$, a contradiction.

    The final conclusion follows from another application of Proposition~\ref{prop:frank-lieb}.
\end{proof}

\section{Applications}
\label{sec:applications}

In this section we discuss applications of Theorem~\ref{thm:michael-simon-sphere}.

Our first set of applications is to sharp $L^2$-Michael--Simon inequalities.
We begin with the general statement about submanifolds of simply-connected manifolds of constant sectional curvature.

\begin{proof}[Proof of Corollary~\ref{cor:michael-simon-spaceform}]
    Every simply-connected spaceform is conformally develop\-able~\cite{Kuiper1949}.
    The conclusion readily follows from Equation~\eqref{eqn:q-on-einstein} and Theorem~\ref{thm:michael-simon-sphere}.
\end{proof}

In particular, we obtain the sharp $L^2$-Michael--Simon inequality.

\begin{proof}[Proof of Theorem~\ref{thm:sharp-michael-simon}]
    Apply Corollary~\ref{cor:michael-simon-spaceform} to $(\mathbb{R}^N,dx^2)$.
\end{proof}

We also obtain the sharp Sobolev inequality on minimal submanifolds of simply-connected spaceforms.

\begin{proof}[Proof of Corollary~\ref{cor:minimal-michael-simon}]
    Apply Corollary~\ref{cor:michael-simon-spaceform} with $H=0$.
\end{proof}

Our second set of applications involve sharp geometric inequalities for conformal immersions into the round sphere.
We begin with a sharp inequality involving the $L^2$-norm of $H$.

\begin{proof}[Proof of Corollary~\ref{cor:willmore}]
    Suppose first that $n \geq 3$.
    Since $\Sigma$ is compact, taking $u=1$ in Theorem~\ref{thm:michael-simon-sphere} yields
    \begin{equation*}
        \int_\Sigma Q \dvol \geq \frac{n}{2} \Vol(S^n)^{\frac{2}{n}} \Vol(\Sigma)^{\frac{n-2}{n}}
    \end{equation*}
    with equality if and only if $j$ is umbilic and $g$ is round.

    Suppose now that $n = 2$.
    By composing with stereographic projection, we may regard $j$ as a conformal immersion $j \colon (\Sigma^2,g) \to (\bR^N,dx^2)$.
    Since $\Sigma$ is compact, integrating Equation~\eqref{eqn:q-transformation} yields
    \begin{equation*}
        \int_\Sigma Q \dvol = \int_\Sigma Q^{j^\ast dx^2} \dvol_{j^\ast dx^2} = \int_\Sigma \lv H \rv^2 \dvol_{j^\ast dx^2} .
    \end{equation*}
    The conclusion follows from a result of Li and Yau~\cite{LiYau1982}*{Theorem~6}.
\end{proof}

We now consider a sharp inequality involving the $L^{n/2}$-norm of $H$.

\begin{proof}[Proof of Corollary~\ref{cor:chen-willmore}]
    Recall that $Q \leq Q_+$.
    The conclusion follows trivially from Corollary~\ref{cor:willmore} if $n=2$.
    Suppose now that $n \geq 3$.
    Combining Corollary~\ref{cor:willmore} with H\"older's inequality yields
    \begin{equation*}
        \Vol(S^n) \Vol(\Sigma)^{\frac{n-2}{2}} \leq \left( \frac{2}{n}\int_\Sigma Q_+ \dvol \right)^{\frac{n}{2}} \leq \left( \frac{2}{n} \right)^{\frac{n}{2}} \Vol(\Sigma)^{\frac{n-2}{2}} \int_\Sigma Q_+^{n/2} \dvol
    \end{equation*}
    with equality if and only if $(\Sigma,g)$ is round and $j$ is umbilic.
\end{proof}

Our third application, which is instead of the ideas used to prove Theorem~\ref{thm:michael-simon-sphere}, is to slightly streamline the proof of the El Soufi--Ilias eigenvalue estimate.

\begin{proof}[Proof of Theorem~\ref{thm:elsoufi-ilias}]
    Lemma~\ref{lem:balance} implies that, by left-composing $j$ with an element of $\Conf(S^N)$ if necessary, we may assume that $j^\ast x^i$ has mean zero for each $i \in \{ 0 , 1, \dotsc, N \}$.

    Since $j$ is conformal, there is an $\Upsilon \in C^\infty(M)$ such that $g = e^{2\Upsilon}j^\ast d\theta^2$.
    Equations~\eqref{eqn:q-transformation} and~\eqref{eqn:energy-invariance} imply that
    \begin{align*}
        \int_\Sigma Q^g \dvol_g & = \int_\Sigma Q^{j^\ast d\theta^2} \dvol_{j^\ast d\theta^2} , && \text{if $n=2$} , \\
        \frac{n-2}{2}\int_\Sigma Q^g \dvol_g & = \mathcal{E}^g(1) = \mathcal{E}^{j^\ast d\theta^2}\left( e^{\frac{n-2}{2}\Upsilon} \right) , && \text{if $n \geq 3$} ,
    \end{align*}
    respectively.
    Since $Q^{j^\ast d\theta^2} \geq \frac{n}{2}$, we deduce that
    \begin{equation}
        \label{eqn:first-q-lower-bound}
        \int_\Sigma Q^g \dvol_g \geq \frac{n}{2}\int_\Sigma e^{(n-2)\Upsilon} \dvol_{j^\ast d\theta^2} .
    \end{equation}

    Recall now that
    \begin{equation*}
        \sum_{i=0}^N \lvert \nabla (j^\ast x^i) \rvert_{j^\ast d\theta^2}^2 = n .
    \end{equation*}
    Inserting this into Inequality~\eqref{eqn:first-q-lower-bound} and using the definition of $\Upsilon$ yields
    \begin{equation*}
        \int_\Sigma Q^g \dvol_g \geq \frac{1}{2}\sum_{i=0}^N \int_\Sigma \lvert \nabla (j^\ast x^i) \rvert_{j^\ast d\theta^2}^2 e^{(n-2)\Upsilon} \dvol_{j^\ast d\theta^2} = \frac{1}{2} \sum_{i=0}^N \int_\Sigma \lvert \nabla (j^\ast x^i) \rvert_g^2 \dvol_g .
    \end{equation*}
    Since each $j^\ast x^i$ has zero mean, we conclude that
    \begin{equation*}
        \int_\Sigma Q^g \dvol_g \geq \frac{\lambda_1}{2}\Vol_g(\Sigma) . \qedhere
    \end{equation*}
\end{proof}

\section{Optimality of the extrinsic conformal Laplacian}
\label{sec:sharpness}

In this section we prove that Theorem~\ref{thm:michael-simon-sphere} is optimal, in the sense that one cannot replace the extrinsic conformal Laplacian by a smaller operator.
In order to make this precise, we first classify the space of differential operators with the same properties as $L$.
To that end, recall that a \defn{natural} linear differential operator on isometric immersions $j \colon (\Sigma^n,g) \to (M^N,h)$ is an assignment $D$ to each such immersion of a linear differential operator $D \colon C^\infty(\Sigma) \to C^\infty(\Sigma)$ by a universal linear combination of complete contractions of tensor products of the metrics on the cotangent and conormal bundles, projections of covariant derivatives of the Riemann curvature tensor of $h$, covariant derivatives of the second fundamental form, and covariant derivatives of the input function with respect to $g$;
see~\cite{GrahamKuo2024} for further discussion.

\begin{lemma}
    \label{lem:operators}
    Fix $N > n \geq 2$.
    Denote by $\sD_{n,N}$ the real vector space of natural linear differential operators defined on isometric immersions $j \colon (\Sigma^n,g) \to (M^N,h)$ that satisfy the conformal transformation law~\eqref{eqn:conformal-transformation}.
    \begin{enumerate}
        \item If $N \geq n+2$, then $\{ L , \lv \tfsff \rv^2 , \tr^2 j^\ast W \}$ is a basis for $\sD_{n,N}$.
        \item If $N = n+1$ or if we restrict to conformal immersions into locally conformally flat manifolds, then $\{ L , \lv \tfsff \rv^2 \}$ is a basis for $\sD_{n,N}$.
    \end{enumerate}
\end{lemma}

\begin{proof}
    Let $D \in \sD_{n,N}$.
    Elementary invariant theory yields a constant $c$ and a scalar Riemannian invariant $I$ of homogeneity $-2$ such that $D = c\Delta + I$.
    Therefore $D - cL$ is a scalar conformal invariant of weight $-2$.
    Arguing as in~\cite{CaseGrahamKuoTyrrellWaldron2023a}*{Subsection~6.1}, we see that if $N \geq n+2$, then the space of such scalar conformal invariants is two-dimensional and spanned by $\{ \lv \tfsff \rv^2 , \tr^2 j^\ast W \}$.
    Clearly $\tr^2 j^\ast W$ vanishes when $j$ is an immersion into a locally conformally flat manifold.
    Since $W$ is a section of $S^2\Lambda^2T^\ast M$ and $\tr^2W = 0$ in $M$, we also see that $\tr^2 j^\ast W$ vanishes when $j$ has codimension one.
\end{proof}

In particular, any conformally invariant $L^2$-Michael--Simons inequality for conformal immersions into the round sphere must be stated in terms of the operator $L_\varepsilon := L + \varepsilon \lv \tfsff \rv^2$ for some $\varepsilon \in \bR$.
Since $\lv \tfsff \rv^2$ is nonnegative, we see that if $\varepsilon \geq 0$, then $L_\varepsilon \geq L$;
i.e.
\begin{equation*}
    \int_\Sigma uL_\varepsilon u \dvol \geq \int_\Sigma uLu \dvol
\end{equation*}
for all $u \in \CC^\infty(\Sigma)$.
In particular, the $L^2$-Michael--Simons-type inequality~\eqref{eqn:michael-simon-epsilon} follows immediately from Theorem~\ref{thm:michael-simon-sphere} when $\varepsilon \geq 0$.
This conclusion is sharp in $\varepsilon$.

\begin{proof}[Proof of Theorem~\ref{thm:michael-simon-super-sharp}]
    It suffices to show that if $\varepsilon < 0$, then there is an isometric immersion $j \colon (\Sigma^n,g) \to (\bR^{n+1},dx^2)$ such that
    \begin{equation*}
        \inf \left\{ \int_\Sigma \left( uLu + \varepsilon\lv\tfsff\rv^2u^2 \right) \dvol \suchthat u \in \CC^\infty(\Sigma) , \int_\Sigma u^2 \dvol = 1 \right\} < 0 .
    \end{equation*}

    Set $\Sigma^n := S^1 \times \bR \times \bR^{n-2}$ and define $j \colon \Sigma \to \bR^{n+1}$ by
    \begin{equation*}
        j(\theta,t,y) := (\cosh t \cos\theta , \cosh t \sin\theta , t, y) ;
    \end{equation*}
    i.e.\ $j(\Sigma)$ is the product of a catenoid with $\bR^{n-2}$.
    Direct computation implies that $j$ is a minimal immersion and that the metric $g := j^\ast dx^2$ and the squared length of the trace-free part of the second fundamental form of $j$ are
    \begin{align*}
        g & = (\cosh t)^2( d\theta^2 + dt^2) + dy^2 , \\
        \lv \tfsff \rv^2 & = 2\sech^4t ,
    \end{align*}
    respectively.
    It readily follows that
    \begin{equation}
        \label{eqn:catenoid-energy}
        \int_\Sigma u L_\varepsilon u \dvol = \int_0^{2\pi} \int_\bR \int_{\bR^{n-2}} \bigl( u_\theta^2 + u_t^2 + u_y^2\cosh^2t + 2\varepsilon u^2 \sech^2t \bigr) \, dy \, dt \, d\theta
    \end{equation}
    for all $u \in \CC^\infty(\Sigma)$, where $L_\varepsilon := L + \varepsilon \lv \tfsff \rv^2$.
    Given $R,T>0$, pick $v_R \in \CC^\infty(\bR^{n-2})$ and $w_T \in \CC^\infty(\bR)$ such that
    \begin{align*}
        v_R & \in [0,1] , & \text{$v_R=1$ on $B_R(0)$} , && \text{$v_R=0$ on $\bR^{n-2} \setminus B_{2R}(0)$} , && \lvert \nabla v_R \rvert \leq 2R^{-1} , \\
        w_T & \in [0,1] , & \text{$w_T=1$ on $B_T(0)$} , && \text{$w_T=0$ on $\bR \setminus B_{2T}(0)$} , && \lvert \nabla w_T \rvert \leq 2T^{-1} ,
    \end{align*}
    where all geometric quantities are measured using the flat metric.
    Set $u := v_Rw_T$.
    On the one hand, we readily estimate that
    \begin{equation*}
        \int_0^{2\pi} \int_\bR \int_{\bR^{n-2}} \bigl( u_\theta^2 + u_t^2 + u_y^2\cosh^2t \bigr) \, dy \, dt \, d\theta \leq C\bigl( T^{-1}R^{n-2} + e^{2T}R^{n-4} \bigr)
    \end{equation*}
    for some constant $C>0$ that is independent of $R$ and $T$.
    On the other hand, since $\int_{-1}^1 \sech^2t \, dt > 0$, we see that there is a constant $C'>0$ such that
    \begin{equation*}
        \int_0^{2\pi} \int_\bR \int_{\bR^{n-2}} u^2\sech^2t \, dy \, dt \, d\theta \geq C'R^{n-2}
    \end{equation*}
    for all $R>0$ and $T>1$.
    Inserting these into Equation~\eqref{eqn:catenoid-energy} yields
    \begin{equation*}
        \int_\Sigma u L_\varepsilon u \dvol \leq 2\varepsilon C' R^{n-2} + CT^{-1}R^{n-2} + Ce^{2T}R^{n-4}
    \end{equation*}
    for all $R>0$ and $T>1$.
    This becomes negative for $R = e^{2T} \gg 1$.
\end{proof}

\section*{Declarations}

\subsection*{Acknowledgements}
We thank Matthew J.\ Gursky for suggesting we consider Theorem~\ref{thm:michael-simon-super-sharp}.

\subsection*{Funding}
JSC was partially supported by a Simons Foundation Collaboration Grant for
Mathematicians.
JSC and DM were partially supported by the National Science Foundation under Award No. DMS-2505606.

\subsection*{AI Statement}
AI was not used in any part of this work.
The authors developed the mathematical arguments and wrote the manuscript.

\bibliography{bib}
\end{document}